\documentclass[11pt]{article}
\usepackage{amsmath,amssymb,amsthm,latexsym,color,epsfig,a4, enumerate}
\usepackage{a4wide}

\newtheorem{theorem}{Theorem}[section]
\newtheorem{lemma}[theorem]{Lemma}
\newtheorem{claim}[theorem]{Claim}
\newtheorem{proposition}[theorem]{Proposition}
\newtheorem{observation}[theorem]{Observation}
\newtheorem{corollary}[theorem]{Corollary}

\newtheorem{definition}[theorem]{Definition}
 
\def\cF{{\cal F}}

\def\cK{{\cal K}}

\newcommand{\gnp}{\ensuremath{\mathcal{G}_{n,p}}}
\newcommand{\Exp}{\mbox{\bf E}}
\renewenvironment{proof}{\noindent{\bf Proof\,}}{\hfill$\Box$}



\title{A remark on the Tournament game}
\author{Dennis Clemens \thanks {Technische Universit\"at Hamburg-Harburg, Institut f\"ur Mathematik, Am Schwarzenberg-Campus 3, 21073 Hamburg, Germany. Email: dennis.clemens@tuhh.de} \and Mirjana Mikala\v{c}ki \thanks{Department of Mathematics and Informatics, Faculty of Sciences, University of
Novi Sad, Serbia. Research partly supported by Ministry of Education
and Science, Republic of Serbia, and Provincial Secretariat for
Science, Province of Vojvodina. Email: mirjana.mikalacki@dmi.uns.ac.rs}}
\date{}

\begin{document}
\maketitle
\begin{abstract}
We study the Maker-Breaker tournament game played on the edge set of a given graph $G$. Two players, Maker and Breaker claim unclaimed edges of $G$ in turns, and Maker wins if by the end of the game she claims all the edges of a pre-defined goal tournament. Given a tournament $T_k$ on $k$ vertices, we determine the threshold bias for the $(1:b)$ $T_k$-tournament game on $K_n$. We also look at the $(1:1)$ $T_k$-tournament game played on the edge set of a random graph $\gnp$ and determine the threshold probability for Maker's win. We compare these games with the clique game and discuss whether a random graph intuition is satisfied.
\end{abstract}
\section{Introduction}
Let $X$ be a finite set and let $\cF \subseteq 2^X$ be a family of the subsets of $X$. Let $a$ and $b$ be two positive integers. In the $(a:b)$ Maker-Breaker positional game $(X,\cF)$ two players, Maker and Breaker, take turns in claiming previously unclaimed elements of $X$, with Maker going first. In each turn, Maker claims $a$ unclaimed elements and then Breaker claims $b$ unclaimed elements of $X$. The game is played until all the elements of $X$ are claimed. Maker wins the game if she claims all the elements of some $F\subseteq \cF$ by the end of the game. Otherwise, Breaker wins. If Maker can win against any strategy of Breaker then the game is said to be a \textit{Maker's win}. Otherwise, the game is said to be a \textit{Breaker's win}. The set $X$ is referred to as the \textit{board} of the game, while the elements of $\cF$ are referred to as the \textit{winning sets}. The values $a$ and $b$ are called \textit{biases} of Maker, and Breaker, respectively. The most basic case of these games are \textit{unbiased games}, where $a=b=1$.

In this paper, we focus on Maker-Breaker \textit{graph games}, i.e., games where the board is the edge set of a given graph $G$. 
In these games Maker's aim is to create a graph consisting only of edges claimed by her that contains some predefined graph theoretic structure. For example, in the \textit{$k$-clique game} (or sometimes abbreviated just as \textit{clique game} when the value of $k$ is not crucial), Maker's goal is to create a graph that contains a clique of order at least $k$. We denote this game by $(E(G),\cK_k)$.
  
Here, we study a variant of the clique game - the \textit{$T$-tournament game} $(E(G),\cK_{T})$. In the $T$-tournament game, introduced by Beck in~\cite{BeckBook}, the goal graph is a tournament $T$, a complete graph where each edge is directed. Before the game starts, the tournament $T$ is fixed, and Maker and Breaker in turns claim edges, and Maker also chooses one of the two possible \textit{orientations} whenever she claims an edge. If her graph contains a copy of the given tournament $T$ by the end of the game, Maker wins. Otherwise, Breaker does.

\noindent \textbf{Games on $\mathbf{K_n.}$} Very well-studied graph games are the ones where $G=K_n$ is the complete graph on $n$ vertices. 
Erd\H{o}s and Selfridge~\cite{ES} initiated the study of the largest value of $k$, $k_c=k_c(n)$, such that Maker can win the $k_c$-clique game on $K_n$ and they were able to prove that $k_c\leq 2(1-o(1))\log_2 n$. Indeed, it turns out that in $(1:1)$ Maker-Breaker clique game on $K_n$, Maker has a strategy to occupy a clique of size $(2-o(1))\log_2(n)$, as shown by Beck~\cite{BeckBook}, and therefore $k_c=2(1-o(1))\log_2 n$ holds. The most interesting fact about this result is that it shows an intriguing relation between games and random graphs here, referred to as the {\em random graph intuition} or {\em probabilistic intuition}. To be precise,
if both players played randomly throughout the game, then Maker's graph would be distributed as a random graph with $n$ vertices
and $\lceil \frac{1}{2}\binom{n}{2} \rceil$ edges, which is well known to have clique size $(2-o(1))\log_2(n)$ with high probability, see e.g. \cite{AS}.
That is, for most values of $k$, a randomly played $k$-clique game on $K_n$ typically has the same winner as the deterministic game played by two intelligent players.

For the tournament game we can ask the same question, as initiated by Beck~\cite{BeckBook}. Motivated by the study of a randomly played $T$-tournament game, Beck conjectured that the largest value of $k$, $k_t=k_t(n)$, for which Maker can win in the $T$-tournament game, for any tournament $T$ on (at most) $k$ vertices, is of size $(1-o(1))\log_2 n$.
However, as the first author together with Gebauer and Liebenau ~\cite{CGL} showed, the truth is twice as large as the conjectured value, i.e., $k_t=(2-o(1))\log_2 n$. This, in particular, tells us two things. Opposite to the clique game, the tournament game does not satisfy the random graph intuition mentioned above. Secondly, since the two values, $k_c$ and $k_t$, are very close to each other, it does not make a big difference for Maker whether she needs to build a graph with or without orientations in an unbiased game on $K_n$.

In the following, we want to find out whether we have similar observations in case we fix $k$ to be a constant, while changing either the bias of Breaker or the board of the game. We start with biased games, in order to give more power to Breaker.
Chv\'{a}tal and Erd\H{o}s~\cite{ChEr} observed that Maker-Breaker games are \textit{bias monotone}, meaning that if the $(1:b)$ game $(X,\cF)$ is a Breaker's win, then the $(1:b+1)$ game is also a Breaker's win. Having this in mind, it thus becomes interesting to find the unique  \textit{threshold bias} $b_{\cF}(n)=b_{\cF}$, which is the largest non-negative integer such that for every $b\leq b_{\cF}$ the $(1:b)$ game is a Maker's win. For the $k$-clique game on $K_n$, Bednarska and \L uczak~\cite{BL} showed that the threshold bias is $b_{\cK_k}=\Theta(n^{\frac{2}{k+1}})$. Naturally, one may wonder what happens with the tournament game, and whether in this case orientations of the edges make things more complicated for Maker. We show that, for every tournament $T$ of order $k$, the threshold bias of the $T$-tournament game $(E(K_n),\cK_{T})$ is of the same order as in the $k$-clique game. 

\begin{proposition}
\label{biased_tournament}
Let $T$ be a tournament on $k\geq 3$ vertices, then 
the threshold bias for the $T$-tournament game on $K_n$
is $b_{\cK_{T}}=\Theta(n^{\frac{2}{k+1}})$.
\end{proposition}

\textbf{Games on random boards.} Another way to give Breaker more power in positional games is to play unbiased graph games on a random graph, as introduced by Stojakovi\'{c} and Szab\'{o}~\cite{SS}. The idea behind this approach is to make the board sparser before the game starts by randomly eliminating edges, so that some of the winning sets no longer exist. We look at the random graph model $\gnp$, which is obtained from the complete graph on $n$ vertices by removing each edge independently with probability $1-p$. 

Now, if an unbiased game $(E(K_n),\cF)$ is a Maker's win, then we are curious about finding the \textit{threshold probability} $p_{\cF}$ such that for $p=\omega (p_{\cF})$ the game $(E(\gnp),\cF)$ is a Maker's win asymptotically almost surely (i.e.\ with probability tending to $1$ as $n$ tends to infinity and abbreviated \textit{a.a.s.\ }in the rest of the paper), and for $p=o(p_{\cF})$, the game $(E(\gnp),\cF)$ is a.a.s.\ a Breaker's win.

When the $k$-clique game is played on $\gnp$, Stojakovi\'c and Szab\'o~\cite{SS} showed that for $k=3$, in the \textit{triangle} game, $p_{\cK_3}=n^{-\frac{5}{9}}$ and for $k\geq 4$, it holds that $n^{-\frac{2}{k+1}-\varepsilon} \leq p_{\cK_k}\leq n^{-\frac{2}{k+1}}$. M\"{u}ller and Stojakovi\'c~\cite{MS} recently proved that for all $k\geq 4$ the threshold probability is indeed $p_{\cK_k} = n^{-\frac{2}{k+1}}$. This again underlines an intriguing relation between games and random graphs, again referred to as the \textit{probabilistic intuition}. Indeed, what we can observe here in case $k\geq 4$ (and also holds for several other natural graph games)
is that the threshold probability for Maker's win in the $(1:1)$ game $(E(\gnp),\cF)$ is of the same order of magnitude as the inverse of the threshold bias $b_{\cF}$ in the $(1:b)$ game $(E(K_n),\cF)$. The triangle game is the only exception in this regard, as 
here Maker a.a.s.\ can win also for probabilities below the so-called critical probability $1/b_{\cK_3}$. 

We show that the tournament game behaves similarly to the clique game when played on $\gnp$. So, even when played on a sparse graph $\gnp$, creating a graph with oriented edges is not much more difficult for Maker than creating a graph without oriented edges. For the tournaments on $k$ vertices, $k\geq 4$, we show the following, which also supports the probabilistic intuition.

\begin{proposition}\label{random_tournament}
Let $T$ be a tournament on $k\geq 4$ vertices, then 
the threshold probability for winning the $T$-tournament game on
$\gnp$ is $n^{-\frac{2}{k+1}}.$
\end{proposition}

So again, for $k\geq 4$, the outcome of the game does not depend much on the choice of the tournament $T$ on $k$ vertices, i.e., on the way the edges of the goal tournament are oriented. However, our next theorem states that the tournament on three vertices behaves differently. In case $T$ is the acyclic triangle $T_A$, we obtain the same threshold probability as in the triangle game on $\gnp$. But, in case $T$ is the cyclic triangle $T_C$, the threshold probability is closer to the critical probability $1/b_{\cK_3}$.

\begin{theorem}\label{T_3}
The threshold probability for winning the unbiased $T_A$-tournament game on
$\gnp$ is $p_{\cK_{T_A}}=n^{-\frac{5}{9}}$, while for the unbiased $T_C$-tournament game 
this threshold probability is $p_{\cK_{T_C}}=n^{-\frac{8}{15}}.$
\end{theorem}

{\bf Notation and terminology.}
Our graph-theoretic notation is standard and follows that of~\cite{West}. In par\-ti\-cu\-lar, we
use the following. For a graph $G$, $V(G)$ and $E(G)$ denote its sets of vertices and edges
respectively, $v(G) = |V(G)|$ and $e(G) = |E(G)|$. For disjoint sets $A,B \subseteq V(G)$, let $E(A,B)$ denote the set of edges of $G$ with one endpoint in $A$ and one endpoint in $B$. Given two vertices, $x$ and $y$, an undirected edge is denoted by $xy$, while $(x,y)$ is a directed edge with orientation from vertex $x$ towards vertex $y$. If an edge is unclaimed by any of the players we call it \textit{free}. For a vertex $x \in V(G)$, $N(x) = \{u \in
G : \exists v \in S, uv \in E(G)\}$ denotes the set of neighbours of the vertex $x$ in $G$. We
let $d(x) = |N(x)|$ denote the degree of vertex $x$ in graph $G$. The minimum and maximum degrees of a graph $G$ are denoted by $\delta(G)$ and $\Delta(G)$ respectively. The density of a graph $G$ is defined as $d(G)=\frac{e(G)}{v(G)}$, while \textit{maximum density } is $m(G)=\max_{H\subseteq G} d(H)$. 

Let $n,k \in\mathbb{N}$ be positive integers. Then with $T_{n,k}$ we denote the {\em Tur\'an graph}
with $n$ vertices and $k$ vertex classes.
That is, its vertex set $V(T_{n,k})=[n]$ comes with a partition $V(T_{n,k})=V_1\cup\ldots \cup V_k$ such that
$\Big| |V_i|- |V_j| \Big|\leq 1$ for all $1\leq i<j\leq k$, and such that its edge set is 
$E(T_{n,k})=\{vw\ |\ v\in V_i,\ w\in V_j,\  1\leq i<j\leq k\}.$ 
Moreover, let $G$ be a graph on at most $k$ vertices, then we say that a subgraph $H\subseteq T_{n,k}$ is a {\em good copy}
of $G$ in $T_{n,k}$, if $G\cong H$ and $|V(H)\cap V_i|\leq 1$ for every $i\in [k]$.
Let $p\in [0,1]$ and moreover let $M\in [e(T_{n,r})]$. Then with ${\cal G}(T_{n,k},p)$ we denote the random graph model obtained from $T_{n,k}$ by deleting each edge of $T_{n,k}$ independently with probability $1-p$.
That is, ${\cal G}(T_{n,k},p)$ is the probability space of all subgraphs $G$ of $T_{n,k}$,
where the probability for a subgraph to be chosen is $p^{e(G)}(1-p)^{e(T_{n,k})-e(G)}$.
Similarly, with ${\cal G}(T_{n,k},M)$ we denote the probability space of 
all subgraphs $G$ of $T_{n,k}$ with $M$ edges, together with the uniform distribution.

Let $Bin(n,p)$ denote the binomial distribution, i.e.\ the distribution of the number of successes
among $n$ independent experiments, where in each experiment we have success with probability~$p$.
Moreover, let us write $X\sim Bin(n,p)$ if $X$ is a random variable with distribution $Bin(n,p)$. 

Finally, $W_k=(V,E)$ is called a {\em k-wheel}, if it is obtained from the cycle $C_k$
by adding one further vertex $z$ which is made adjacent to every vertex of $C_k$.
The special vertex $z$ is called the {\em center} of $C_k$.

Throughout the paper $\ln$ stands for the natural logarithm. 

{\bf Organization of the paper.}
The rest of the paper is organized as follows. At first we collect some useful results in the Preliminaries. In Section~\ref{sec:prop} we prove Proposition~\ref{biased_tournament} and Proposition~\ref{random_tournament}. Finally, in Section~\ref{sec:triangle} we prove Theorem~\ref{T_3}.

\section{Preliminaries} \label{tg:preliminaries}

The following estimate is usually referred to as a Chernoff inequality \cite{JLR}.

\begin{lemma}[Theorem 2.1 in \cite{JLR}]\label{chernoff}
Let $X\sim Bin(n,p)$ and $\lambda = \Exp(X) = np$. Then for $t\geq 0$, it holds that
$Pr(X\geq \Exp(X) + t)\leq \exp\Big( -\frac{t^2}{2\lambda} + \frac{t^3}{6\lambda^2} \Big).$
\end{lemma}

As indicated above, we will consider the random graph models
${\cal G}(T_{n,k},p)$ and ${\cal G}(T_{n,k},M)$.
For this, we will make use of some general results about random sets.

Following \cite{JLR}, let $\Gamma$ be a set of size $N\in\mathbb{N}$.
For $p\in [0,1]$, we let $\Gamma_p$ denote the probability space of all subsets $A\subseteq \Gamma$,
where the probability of choosing $A$ is $p^{|A|}(1-p)^{|\Gamma\setminus A|}$.
Moreover, for $M\in [N]$, we let $\Gamma_M$ denote the probability space
of all subsets $A\subseteq \Gamma$ of size $M$,
together with the uniform distribution. In case we choose a random set $A$ according to the model
$\Gamma_p$, we shortly write $A\sim \Gamma_p$. Similarly, we write $A\sim \Gamma_M$,
when $A$ is chosen according to the uniform model $\Gamma_M$.

One important fact about the two models above is that in many cases they are closely related to each other 
when $p\sim \frac{M}{N}$; see Section 1.4 in \cite{JLR}. In particular, we will make use of the following two
statements, which help us to transfer results from one model to the other.

\begin{lemma}[Pittel's Inequality, Equation (1.6) in \cite{JLR}]\label{random_pittel}
Let $\Gamma$ be a set of size $N$, let $M\in [N]$, and $p=\frac{M}{N} \in [0,1]$.
Let ${\cal P}$ be a family of subsets of $\Gamma$. Moreover, let $H_p\sim \Gamma_p$
and $H_M\sim \Gamma_M$, then
$$Pr(H_M\notin {\cal P})\leq 3\sqrt{M}\cdot Pr(H_p\notin {\cal P}).$$
\end{lemma}

\begin{lemma}[Corollary 1.16 (iii) in \cite{JLR}]\label{random_models}
Let $\Gamma$ be a set of size $N$ and let $M\in [N]$.
Let $\delta>0$ be such that $0\leq (1+ \delta)\frac{M}{N} \leq 1$, and let $p=(1+ \delta) \frac{M}{N}$.
Let ${\cal P}$ be a family of subsets of $\Gamma$. Moreover, let $H_p\sim \Gamma_p$
and $H_M\sim \Gamma_M$, then
$$Pr(H_M\in {\cal P})\rightarrow 1 \text{ implies }Pr(H_p\in {\cal P})\rightarrow 1.$$
\end{lemma}

Later we want to know whether a certain random graph contains a copy of a fixed graph
with high probability. In this regard, we make use of the following two theorems.

\begin{theorem}[Theorem 2.18 (ii) in \cite{JLR}]\label{random_bound1}
Let $\Gamma$ be a set, $p\in [0,1]$  and let $H\sim \Gamma_p$. 
Let ${\cal S}$ be a family of subsets of $\Gamma$. Moreover, for every $A\in {\cal S}$
let $I_A$ be the indicator variable which is 1 if $A\subseteq H$, and $0$ otherwise.
Finally, let $X=\sum_{A\in S} I_A$ be the random variable counting the number of elements
of ${\cal S}$ that are contained in $H$. Then
$$Pr(X=0)\leq \exp \Big( - \frac{\Exp(X)^2}{\sum_{A\in {\cal S}} \sum_{\substack{B\in {\cal S}\\ A\cap B\neq \emptyset}} \Exp(I_AI_B)}\Big).$$

\end{theorem}

\begin{theorem}[Theorem 3.4 in \cite{JLR}]\label{threshold_subgraph}
Let $H$ be a graph, and let $X_H$ denote random variable counting the number of copies of $H$
in a random graph $G\sim \gnp$. Then, as $n$ tends to infinity, we have
\begin{align*}
Pr(X_H>0)\rightarrow 
\begin{cases}
0 & \text{ if }p\ll n^{-\frac{1}{m(H)}}\\
1 & \text{ if }p\gg n^{-\frac{1}{m(H)}}.
\end{cases}
\end{align*} 
\end{theorem}

\section{Most tournaments behave like cliques}\label{sec:prop}

The main idea for the proof of the propositions is as follows: Let $G$ be the graph on which the game is to be played.
Let $T$ be the goal tournament with vertices $v_1,\ldots, v_k$.
Then, before the game starts Maker splits the vertex set of $G$ into $k$ parts $V_1,\ldots,V_k$
with $\Big| |V_i|- |V_j| \Big|\leq 1$ for all $1\leq i<j\leq k$,
and she identifies each class $V_i$ with the vertex $v_i$ according to the following rule:
Whenever Maker claims an edge between some classes $V_i$ and $V_j$, she always chooses
the direction of this edge according to the direction of the edge $v_iv_j$ in $T$.
Because of this identification, it then remains to show that Maker has a strategy 
for the usual Maker-Breaker game on $G$ to occupy a copy of $K_k$ with exactly one vertex in each $V_i$.

In order to show that Maker has such a strategy
for this game, we will make use of results from \cite{JLR},
and follow the proof ideas from \cite{BL, SS}.
As most parts are proven analogously to results in the aforementioned publications,
we rather keep our argument short and, whenever possible, we refer back to the known results.
At first, analogously to Theorem 3.9 in \cite{JLR}, we bound the probability that a random graph $G\sim {\cal G}(T_{n,k},p)$ does not contain a good copy of $K_k$.

\begin{claim}\label{cor_p}
Let $k\geq 3$ be a positive integer.
Then there is a constant $c_1=c_1(k)>0$ such that for every large enough $n$ the following is true:
If $n^{-\frac{2}{k+1}}\leq p\leq 4n^{-\frac{2}{k+1}}$
and if $X$ denotes the random variable counting the number
of good copies of $K_k$ in a random graph $G\sim {\cal G}(T_{n,k},p)$, then
$Pr(X=0)\leq \exp (-c_1n^2p).$
\end{claim}

{\bf Proof}
Let $G\sim {\cal G}(T_{n,k},p)$. Let $\cal S$ be the family of good copies of $K_k$ in $T_{n,k}$. 
For each such copy $C_i\in {\cal S}$ let
$I_{C_i}$ be the indicator variable which is 1 if and only if $C_i\subseteq G$. 
By Theorem~\ref{random_bound1},
$$Pr(X=0)\leq \exp\Big( - \frac{(\mathbb{E}(X))^2}{\sum_{C_1}\sum_{C_2:\ E(C_1)\cap E(C_2)\neq \emptyset} \mathbb{E}(I_{C_1}I_{C_2})}\Big).$$ 
The denominator in the above expression can be bounded from above by
\begin{align*}
\sum_{t=2}^k \sum_{C_1\in {\cal S}}\sum_{\substack{C_2\in {\cal S}:\\ \ C_1\cap C_2\cong K_t}} p^{2\binom{k}{2}-\binom{t}{2}}
	\leq &\ \sum_{t=2}^k n^{2k-t}p^{2\binom{k}{2}-\binom{t}{2}}\\
	= &\ \Theta(\mathbb{E}(X)^2)\cdot \sum_{t=2}^k n^{-t}p^{-\binom{t}{2}}\\
	= &\ \Theta(\mathbb{E}(X)^2\cdot n^{-2}p^{-1}) \sum_{t=2}^k \Big(n^{-1}p^{-\frac{t+1}{2}}\Big)^{t-2}\\
 = &\ \Theta(\mathbb{E}(X)^2\cdot n^{-2}p^{-1}),
\end{align*}
where in the last equality we use that $p=\Theta(n^{-\frac{2}{k+1}})$. Thus, the claim follows. \hfill $\Box$

\begin{corollary}\label{cor_M}
Let $k\geq 3$ be a positive integer. 
Then there is a constant $c_1'=c_1'(k)>0$ such that for every large enough $n$ the following is true:
If $M=\lfloor n^{2-\frac{2}{k+1}}\rfloor$
and if $X'$ denotes the random variable counting the number
of good copies of $K_k$ in a random graph $G\sim {\cal G}(T_{n,k},M)$, then
$Pr(X'=0)\leq \exp (-c_1'M).$
\end{corollary}

\begin{proof}
Set $p=\frac{M}{e(T_{n,k})}$ and observe that  $n^{-\frac{2}{k+1}}\leq p\leq 4n^{-\frac{2}{k+1}}$.
The statement now follows by Claim \ref{cor_p} and Lemma \ref{random_pittel}.
\end{proof}

\begin{corollary}\label{cor_resilience}
Let $k\geq 3$ be a positive integer.
Then there is a constant $\delta=\delta(k)>0$ such that for every large enough $n$ 
and $M=2\lfloor n^{2-\frac{2}{k+1}}\rfloor$, a random graph $G\sim {\cal G}(T_{n,k},M)$ 
satisfies the following property a.a.s.:
Every subgraph of $G$ with at least $\lfloor (1-\delta)M \rfloor$
edges contains a good copy of $K_k$.
\end{corollary}

{\bf Proof}
We proceed analogously to \cite{BL}. Let $\delta>0$ such that $\delta-\delta \log(\delta)<c_1'/3$,
with $c_1'$ from Corollary~\ref{cor_M},
and count the number of pairs $(H,H')$ where $H$ is a subgraph of $T_{n,k}$ with $M$ edges
and where $H'\subseteq H$ is a subgraph with $\lfloor (1-\delta)M\rfloor$ edges that does not
contain a good copy of $K_k$. Then 
using Corollary \ref{cor_M} (and simplifying the notation slightly by ignoring floor signs) we obtain that the number of such pairs is at most
\begin{align*}
\exp(-\frac{c_1'M}{2})\binom{e(T_{n,r})}{(1-\delta)M} \binom{e(T_{n,r})-(1-\delta)M}{\delta M}
\leq & \exp(-\frac{c_1'M}{2})\binom{M}{\delta M}\binom{e(T_{n,r})}{M}\\
 \leq & \exp\Big(-\frac{c_1'M}{2}+\delta M(1-\log(\delta))\Big) \binom{e(T_{n,r})}{M}\\
= & o(1)\binom{e(T_{n,r})}{M}. \text{\hspace{4.5cm}$\Box$}
\end{align*}

Using this last corollary, we can start proving the existence of Maker strategies. The following claim is an analogue statement to Theorem 19 in \cite{SS}, and thus its proof is analogous to \cite{SS}.

\begin{claim}\label{maker1}
Let $k\geq 3$ and $n$ be positive integers.
Then there is a constant $c_2=c_2(k)>0$ such that for every $M\geq c_2^{-1} n^{2-\frac{2}{k+1}}$, every $1\leq b\leq c_2Mn^{-2+\frac{2}{k+1}}$,
for a random graph $G\sim {\cal G}(T_{n,k},M)$ the following a.a.s.\ holds:
Maker has a strategy to occupy a good copy of $K_k$ in the $(1:b)$ Maker-Breaker game on $G$.
\end{claim}

\begin{proof}
Choose $\delta=\delta(G)$ according to Corollary \ref{cor_resilience} and let $c_2=\delta/10$.
Maker's strategy is as follows: in each of her moves she chooses an edge from $G$ uniformly at random
among all edges from $G$ that have not been claimed so far by herself. If she chooses an edge that is not claimed by Breaker so far,
she claims this edge. Otherwise, Maker declares her move as a failure and skips it.
Similar to \cite{SS}, we consider the first $M':=2\lfloor n^{2-\frac{2}{k+1}} \rfloor\leq \frac{\delta}{2}\cdot \frac{1}{b+1}M$
rounds of the game.
As only a $\frac{\delta}{2}$-fraction of all edges are claimed in these rounds, the probability for a failure is at most $\frac{\delta}{2}$
in each round. So, the number of failures can be ``upper bounded'' by a binomial random variable $X\sim Bin(M',\frac{\delta}{2})$,
which by Chernoff's inequality (Lemma~\ref{chernoff}) satisfies $Pr(X\geq 2\Exp(X))\leq \exp(-\frac{\Exp(X)}{3})=o(1)$.
That is, the number of failures will be at most $\delta M'$ a.a.s.\
Thus, Maker a.a.s.\ creates a graph $H\setminus R$ with $H\sim {\cal G}(T_{n,k},M')$ and $e(R)\leq \delta M'$, against
any strategy of Breaker,
which by Corollary \ref{cor_resilience} a.a.s.\ contains a good copy of $K_k$. Thus, a.a.s.\ Breaker cannot have a strategy to prevent good copies of $K_k$, and as either Maker or Breaker needs to have a winning strategy, the claim follows.
\end{proof}

\begin{corollary}\label{maker2}
Let $k\geq 3$ and $n$ be positive integers
Then there is a constant $c_3=c_3(k)>0$ such that for every $p\geq c_3n^{-\frac{2}{k+1}}$
and $G\sim {\cal G}(T_{n,k},p)$ the following a.a.s.\ holds:
Maker has a strategy to occupy a good copy of $K_k$ in the unbiased Maker-Breaker game on $G$.
\end{corollary}

\begin{proof}
The statement follows immediately from Corollary \ref{maker1} and Lemma \ref{random_models},
where we choose ${\cal P}$ to be the family of all graphs $G\subseteq T_{n,k}$
for which Maker has a strategy to occupy a good copy of $K_k$ in the unbiased Maker-Breaker game
on $E(G)$.
\end{proof}

Finally, we can prove the two propositions.

{\bf Proof of Proposition \ref{biased_tournament}.}
Let $T$ be the tournament, with $k\geq 3$ vertices, of which Maker aims to create a copy on $K_n$.
By Theorem 1 in \cite{BL}, we know that there is a constant $c>0$ such that for large enough $n$ and for every $b\geq cn^{\frac{2}{k+1}}$,
Breaker has a strategy to prevent cliques of order $k$. Using this strategy, Breaker
wins the $T$-tournament game on $K_n$.
Now, let $c_2=c_2(k)$ be given according to Claim \ref{maker1}, and let $M=e(T_{n,k})$, $b=0.25c_2n^{\frac{2}{k+1}}$. Then Claim \ref{maker1} implies that Maker has a strategy to occupy a good copy of $K_k$
in the $(1:b)$ Maker-Breaker game on $T_{n,k}$. But, as we argued earlier, this also gives Maker a strategy 
for the $(1:b)$ $T$-tournament game on $K_n$.
\hfill $\Box$

{\bf Proof of Proposition \ref{random_tournament}.}
Let $T$ be the tournament, with $k\geq 4$ vertices, of which Maker aims to create a copy in an unbiased game on $G\sim \gnp$.
By Theorem 1.1 in \cite{MS}, we know that there is a constant $c>0$ such that for $p\leq cn^{-\frac{2}{k+1}}$,
Breaker a.a.s.\ has a strategy to block cliques of order $k$ in the unbiased Maker-Breaker game on $G$,
which again gives a winning strategy for Breaker in the $T$-tournament game on $G$.
Now, let $p\geq c_3n^{-\frac{2}{k+1}}$, with $c_3=c_3(k)$ from Corollary \ref{maker2}.
Before sampling the random graph $G\sim {\cal G}_{n,p}$ fix a partition $V_1\cup\ldots \cup V_k=[n]$ as before. Then, after sampling 
$G\sim {\cal G}_{n,p}$, we know that the subgraph induced by those edges which intersect two different parts $V_i$ and $V_j$ is sampled like a random graph $F\sim {\cal G}(T_{n,k},p)$. According to Corollary \ref{maker2},
Maker a.a.s.\ has a strategy to occupy a good copy of $K_k$ in $F\subseteq G$,
and thus Maker a.a.s.\ has a strategy to create a copy $T$ in the unbiased tournament game on $G$.
\hfill $\Box$

\section{The triangle case}\label{sec:triangle}

In the following we prove {\bf Theorem \ref{T_3}}.

For the acyclic triangle $T_A$, the result can be obtained from \cite{SS} as follows: For $p\ll n^{-\frac{5}{9}}$
Breaker a.a.s.\ has a strategy to prevent triangles in the unbiased Maker-Breaker game on $G\sim \gnp$.
Applying such a strategy in the $T_A$-tournament game as Breaker obviously blocks acyclic triangles. 
For $p\gg n^{-\frac{5}{9}}$ a.a.s.\ Maker has a strategy to gain an undirected triangle in the unbiased Maker-Breaker game on $G\sim\gnp$. In the $T_A$-game, Maker now can proceed as follows. She fixes an arbitrary ordering $\{v_1,\ldots, v_n\}$ of $V(G)$ before the game starts. Then she applies the mentioned strategy of Maker for gaining an undirected triangle, where she always chooses orientations from vertices of smaller index to vertices of larger index. This way, every triangle claimed by her will be an acyclic triangle, and thus she wins.

Thus, from now on, we can restrict the problem to the discussion of the cyclic triangle $T_C$. To show that $n^{-\frac{8}{15}}$ is the threshold probability for the existence of a winning strategy for Maker in the $T_C$-tournament  game on $G\sim\gnp$, we will study Maker's and Breaker's strategy separately.

\medskip

\begin{figure} [htbp]
\centering
\includegraphics[scale=0.8]{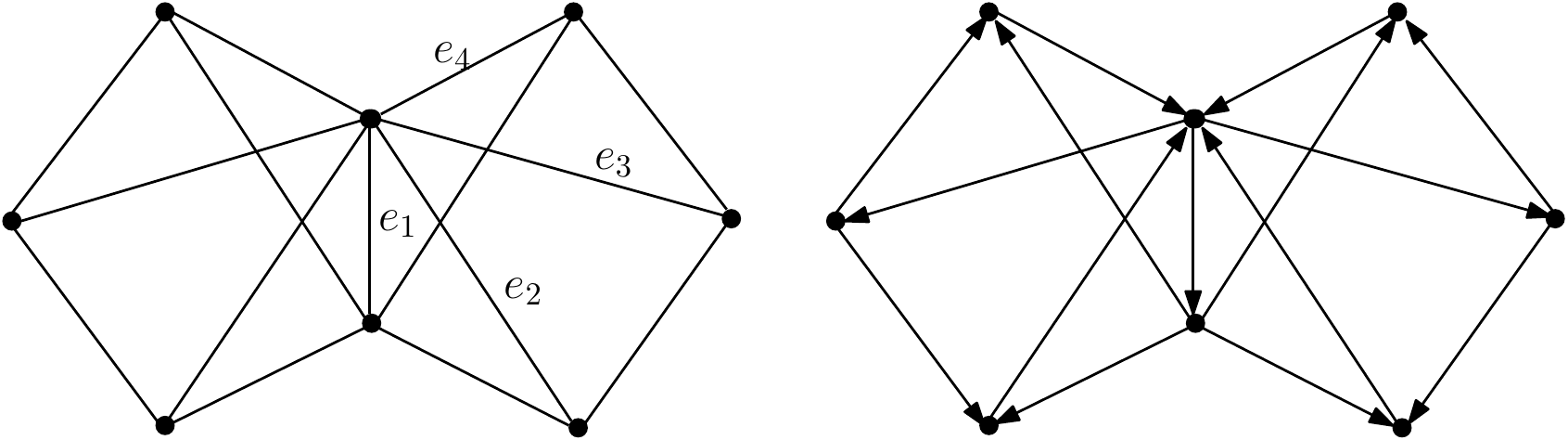}
\caption{Graph $H$ without and with orientation.}
\label{graphH}
 \end{figure}

We start with {\bf Maker's strategy}. Let $p \gg n^{-\frac{8}{15}}$. Then, by Theorem \ref{threshold_subgraph}, a.a.s.\ $G\sim\gnp$
contains the graph $H$, presented in the left half of Figure \ref{graphH}, as $m(H)=\frac{15}{8}$. As indicated in the right half of the same figure, its edges can be oriented in such a way that each triangle has a cyclic orientation, and thus,
it is enough to prove that Maker has a strategy 
to claim an undirected triangle in the unbiased Maker-Breaker game on $H$. 
Her strategy is as follows. At first she claims the edge $e_1$, as indicated in the figure. By symmetry, we can assume 
that afterwards Breaker claims an edge which is on the ``left side" of $e_1$. Then in the next moves, as long
as she cannot close a triangle, 
Maker claims the edges $e_2$, $e_3$ and $e_4$, always forcing Breaker to block an edge
which could close a triangle, and Maker will surely be able to complete a triangle in the next round.

Now, let $p\ll n^{-\frac{8}{15}}$. We are going to show that a.a.s.\ there exists a {\bf Breaker's strategy} which blocks
copies of $T_C$, when playing on $G\sim \gnp$. We start with some preparations. Amongst others, we will consider {\em triangle collections},
as studied in \cite{SS}.

\begin{definition}\label{def:collections}
Let $G=(V,E)$ be some graph without isolated vertices. Further, let \linebreak
$T_G=(V_T,E_T)$ be the graph where
$V_T=\{H\subseteq G:\ H\cong K_3\}$ is the set of all triangles in $G$,
and $E_T=\{H_1H_2:\ E(H_1)\cap E(H_2)\neq \emptyset\}$
is the (binary) relation on $V_T$ of having a common edge. Then:
\begin{itemize}
\item $G$ is called {\em very basic} if $T_G$ is a subgraph of a copy of $K_3^+$ (triangle plus a pending edge), 
or a subgraph of a copy of $P_k$ with $k\in\mathbb{N}$.
\item $G$ is called {\em basic} if there are distinct edges $e_1,e_2\in E(G)$
such that $T_{G-e_i}$ is very basic for both $i\in \{1,2\}$. 
\item $G$ is a triangle collection if every edge of $G$ is contained in some triangle
and $T_G$ is connected.
\end{itemize}
If $G$ is a triangle collection we further call it a bunch (of triangles) if 
we can find triangles $F_1,\ldots F_r\in V_T$ covering all edges of $G$
with the property that $|V(F_i)\setminus \cup_{j<i} V(F_j)|=1$
and $|E(F_i)\setminus \cup_{j<i} E(F_j)|\geq 2$ for every $i\in[r]$.
\end{definition}

\begin{figure} [htbp]
\centering
\includegraphics[scale=0.75]{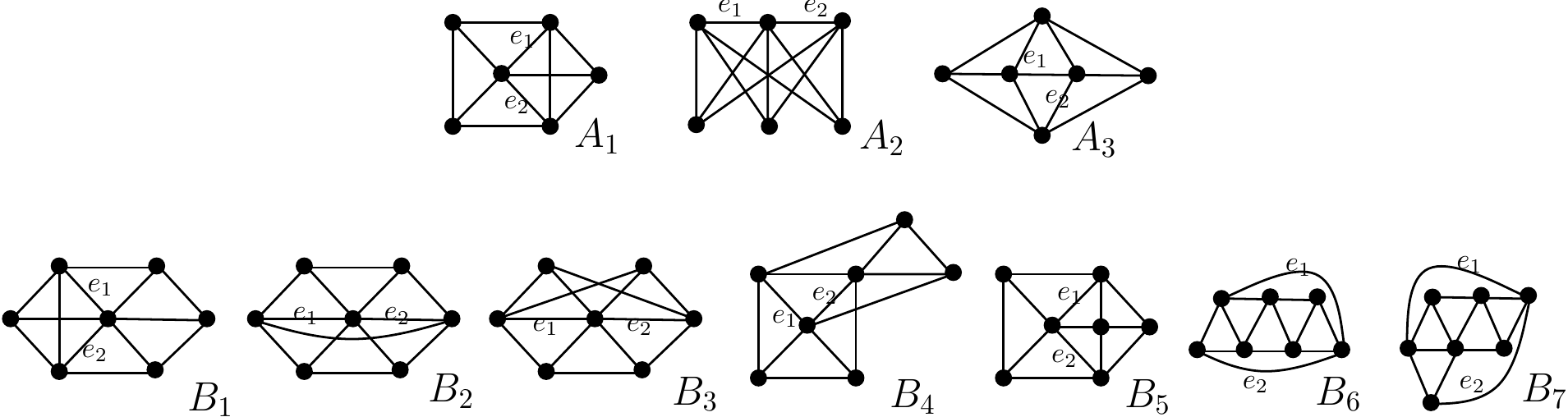}
\caption{Basic triangle collections.}
\label{basic}
 \end{figure}

Note that every collection on a given number $n$
of vertices, contains a bunch on the same number of vertices
with at least $2n-3$ edges. 
Figure \ref{basic} shows some collections that are easily checked to be basic. 
For each of the graphs, the edges $e_1$ and $e_2$ indicated in the figure satisfy the condition from the definition
of basic graphs. Moreover,
the following observation is easily verified.

\begin{observation}\label{reduction}
Let $G=(V,E)$. Maker has a strategy to create a triangle (a copy of $T_C$) on $G$
if and only if $G$ contains a collection $C$ such that she has a strategy to create a triangle (a copy of $T_C$)  on $C$.
\end{observation}
%

In the following we show now that Breaker can prevent Maker from occupying a triangle when playing
on basic graphs. This also ensures a winning strategy for Breaker in the corresponding
$T_C$-tournament game. We start with the following proposition.

\begin{proposition}\label{obs:basic}
Let $G=(V,E)$ be very basic, then Breaker can block every triangle
in the unbiased Maker-Breaker game on $E(G)$, even if Maker is allowed to claim two edges in the very first round.
\end{proposition}

\begin{proof}
Without loss of generality (abbreviated \textit{W.l.o.g.} in the rest of the paper) we can assume that $T_G\cong P_k$ for some $k$, or $T_G\cong K_3^+$,
with $T_G$ as given in Definition \ref{def:collections}.
We further can assume that Maker in the first round claims
two edges $f_1,f_2\in E(G)$ that participate in triangles of $G$.
If $T_G\cong P_k$ then observe that there is an ordering $F_1,\ldots, F_k$ of the elements in $T_G$,
such that $f_1\in E(F_1)$, and
$|V(F_i)\setminus \cup_{j<i} V(F_j)|=1$,
and $|E(F_i)\setminus \cup_{j<i} E(F_j)|=2$ for every $2\leq i\leq k$.
To see this one just has to start the sequence with a triangle $F_1$
containing $f_1$, and to extend the sequence along the path-like structure of $T_G$.
Finally, let $A_1:=E(F_1)\setminus \{f_1\}$
and $A_i:=E(F_i)\setminus \cup_{j<i} E(F_j)$ for every $i\in [k]\setminus\{1\}$.
These sets are pairwise disjoint, have cardinality $2$ and satisfy $A_i\subseteq E(F_i)$ 
for each $i\in [k]$. That is, Breaker can block triangles by an easy pairing strategy.
(In particular, for his first move, Breaker claims the unique edge $f$ 
for which there is an $i\in [k]$ with $A_i=\{f_2,f\}$.) 
If $T_G\cong K_3^+$, then it can be shown that $G$ contains exactly four triangles
and that one can find an ordering $F_1,\ldots F_k$ (with $k=4$) 
with the properties from the previous case. So, Breaker wins similarly.
\end{proof}

\begin{corollary}\label{cor:basic}
Let $G=(V,E)$ be basic, then Breaker can block every triangle
in the unbiased Maker-Breaker game on $E(G)$.
\end{corollary}

\begin{proof}
Let $e_1,e_2$ be the edges given by the definition of a basic graph.
Breaker's strategy is to claim $e_1$ or $e_2$ in the first round.
Afterwards, the game reduces to the graph $G-e_i$ for some $i\in [2]$,
where Maker claims 2 edges, before Breaker claims his first edge.
Now, since $G-e_i$ is very basic for both $i\in \{1,2\}$, Breaker then succeeds by the previous proposition.
\end{proof}

We further observe the following two statements which can be checked
by easy case distinctions.

\begin{observation}\label{K_4}
Breaker has a strategy to prevent cyclic triangles in an unbiased game on $E(K_4)$,
even if Maker is allowed to claim and orient two edges in her first turn. 
\end{observation}


\begin{observation}\label{4-wheel}
Breaker has a strategy to prevent cyclic triangles in an unbiased game on $E(W_4)$,
even if Maker is allowed to claim and orient two edges in her first turn, as long as
not both edges are incident with the center vertex of $W_4$. 
\end{observation}

Now, using the previous
statements we will show that for $p\ll n^{-\frac{8}{15}}$ a.a.s.\ every collection~$C$ in $G\sim G_{n,p}$
is such that Breaker has a strategy to prevent cyclic triangles in an unbiased game on $C$.
It follows then by Observation \ref{reduction} that a.a.s.\ Breaker wins on $G$.
To do so, we start with the following propositions, motivated by \cite{SS}, which helps to restrict
the set of collections we need to consider.

\begin{proposition}
\label{collection_density}
Let $p\ll n^{-\frac{8}{15}}$, then a.a.s.\ every triangle collection $C$ in $G\sim \gnp$.
 satisfies $m(C)<\frac{15}{8}$. 
\end{proposition}

\begin{proof}
Each collection $C$ on at least 25 vertices contains a bunch $B$ on exactly 25 vertices with
$$d(B)=\frac{e(B)}{v(B)}\geq \frac{2v(B)-3}{v(B)}>\frac{15}{8}.$$
Since there are only finitely many such bunches and each of them a.a.s.\ does not appear 
in $G$ according to Theorem \ref{threshold_subgraph}, together with the union bound 
we obtain that a.a.s.\ each collection in $G$ lives on at most $25$
vertices. Since there are only finitely many collections with at most 25 vertices, 
we also know by the same reason that a.a.s.\ each collection  in $G$ on at most 25 vertices
 needs to have maximum density smaller than $\frac{15}{8}$.
\end{proof}

\begin{proposition}\label{properties}
Let $C$ be a triangle collection with $m(C)<\frac{15}{8}$ such 
that Maker has a strategy to create a cyclic triangle in an unbiased game on $C$,
but there is no such strategy for any collection $C'\subset C$.
Then the following properties hold:
\begin{enumerate}[(a)]
\item $5\leq v(C) \leq 7$, 
\item $e(C)= 2v(C)-1$,
\item $\delta(C)\geq 3$,
\item $C$ is not basic.
\end{enumerate}
\end{proposition}

\begin{proof}
Property~(d) obviously holds, using Corollary \ref{cor:basic}. Moreover, (c) follows immediately. Indeed, 
if there were a vertex $v$ with $d_C(v)\leq 2$,
then Breaker could prevent cycles on $C-v$ by the minimality condition on $C$,
and cycles containing $v$ by simply pairing the edges incident with $v$ (if there exist two such edges),
a contradiction. Furthermore, $v(C)\geq 5$ is needed,
according to Observation \ref{K_4}. 
Now, let $B$ be a bunch contained in $C$ with $v(C)$ vertices, then $e(C)>e(B)$, 
since $\delta(B)=2<\delta(C)$. 
As such a bunch contains at least $2v(B)-3$ edges,
it follows that $e(C)\geq e(B)+1 \geq 2v(C)-2$. Furthermore
$e(C)\leq 2v(C)-1$, since otherwise $m(C)\geq 2$. 
If $e(C)=2v(C)-1$, then together with $m(C)<\frac{15}{8}$, we deduce that $v(C)\leq 7$.
Otherwise, we have $e(C)=2v(C)-2$ and $e(C)=e(B)+1$. Analogously to the proof of Theorem 23 in \cite{SS} 
it then follows that $C$ can only be a wheel; for completeness let us include the argument here:
Let $E(C)\setminus E(B)=\{v_1v_2\}$.
By the definition of a bunch, we can find triangles $F_1,\ldots F_r$ in $B$ covering all edges of $B$
with the property that $|V(F_i)\setminus \cup_{j<i} V(F_j)|=1$
and $|E(F_i)\setminus \cup_{j<i} E(F_j)|\geq 2$ for every $i\in[r]$.
As $e(B)=e(C)-1=2v(B)-3$ it then follows that $r=v(C)-2$
and $|E(F_i)\setminus \cup_{j<i} E(F_j)|= 2$ for every $i\in[r]\setminus\{1\}$,
as otherwise $e(B)> 3 + 2(r-1) = 2v(C)-3$, a contradiction. Thus,
for every $i \in[r]\setminus\{1\}$, $F_i$ needs to share exactly one edge with
$\cup_{j<i} F_j$. From this, we can conclude that $B$ needs to contain at least
two vertices of degree 2. However, as $\delta(C)\geq 3$
and $E(C)\setminus E(B)=\{v_1v_2\}$, we know that $v_1$ and $v_2$
must be the only vertices in $B$ of degree 2.
Now, by the definition of a triangle collection, $v_1v_2$
needs to be part of a triangle in $C$. Thus, there needs to be a vertex $v_3$
such that $v_1v_3,v_3v_2\in E(B)$. But this is only possible if $v_3$
belongs to every triangle $F_i$, $i\in[r]$, and thus, $C$ needs to be a wheel.  
Now, to finish the proof, observe that Breaker can always prevent triangles
in an unbiased game on a wheel by a simple pairing strategy,
a contradiction to our assumption.
\end{proof}

So, the goal will be to show that there exists no collection $C$ which satisfies all the conditions
given in Proposition \ref{properties}.

\begin{figure} [htbp]
\centering
\includegraphics[scale=0.7]{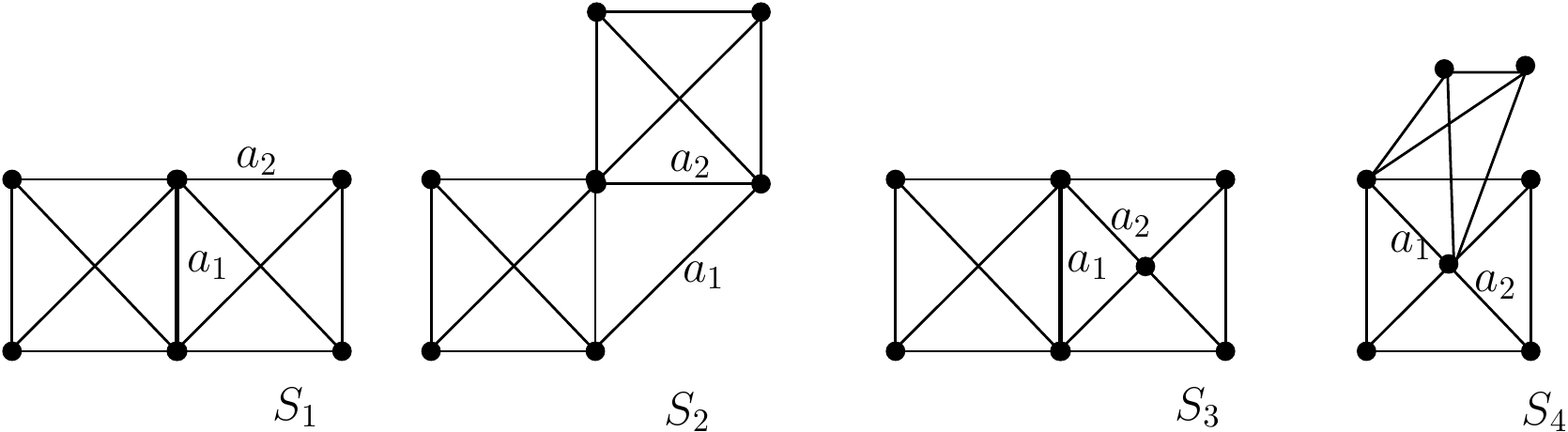}
\caption{Special collections.}
\label{special}
 \end{figure}

\begin{lemma}\label{collection}
If a collection $C$ satisfies (a) - (d) from Proposition \ref{properties}, then
either $C$  is isomorphic to $K_5^-$ ($K_5$ minus one edge) or $C$ is isomorphic to one of the graphs $S_i$, $1\leq i\leq 4$, given in Figure \ref{special}.
\end{lemma}

\begin{proof}
If $v(C)=5$, then $e(C)=9$, by Property (b), and the statement follows obviously. So, let $v(C)\neq 5$.
We will show now that a collection satisfying (a) - (c) either is isomorphic
to one of the collections $S_i$, or it is isomorphic to one
of the basic collections $A_i$ or $B_i$ from Figure \ref{basic}, thus contradicting Property (d).

Let us start with $v(C)=6$.
Assume first that $C$ contains a subgraph $H \cong K_4$ and let $\{x,y\}=V(C)\setminus V(H)$. 
With $e(C)=11$ and $\delta(C)\geq 3$ we conclude
$xy\in E(C)$, and by the definition of a collection it follows that $x$ and $y$ have a common neighbour $v_1\in V(H)$. Because of (c),
we further have $xv_2\in E(C)$ for some $v_2\in V(H)\setminus \{v_1\}$. Now, if $yv_2\in E(C)$,
then $C\cong S_1$, otherwise by (c) we have $yv_3\in E(C)$ for some $v_3\in V(H)\setminus \{v_1,v_2\}$
and so $C\cong A_1$.\\
Assume then that $C$ does not contain a clique of order 4. We still find a subgraph $H'\subseteq C$
with four vertices $V(H')=\{v_1,v_2,v_3,v_4\}$ and five edges, say $v_1v_3\notin E(H)$.
Since $C$ is a triangle collection, there needs to be some $x\in V(C)\setminus V(H')$ that is part of the same triangle
as an edge $e$ from $H'$. Let $y$ be the unique vertex 
in $V(C)\setminus(V(H')\cup \{x\})$.

Assume first that $e=v_2v_4$.  We know then that $\{x,v_1,v_3\}$
is an independent set in $C$, since otherwise we would have a 4-clique in $C$. By (b) and (c), it thus follows
that \linebreak $N(y)=\{x,v_1,v_3,v_i\}$ for some $i\in\{2,4\}$, which gives $C\cong A_2$.

Assume then that $e\neq v_2v_4$ and w.l.o.g.\ $e=v_3v_4$ by symmetry of $H'$.
If $v_1x\in E(C)$, it then follows that $d(y)=3$, since (b) and (c) need to hold;
moreover, $C[V(C)\setminus \{y\}]\cong W_4$ where $v_4$ represents the center of the wheel.
In case $v_4y\in E(C)$, we can only have $C\cong A_2$, as $C$ does not contain a 4-clique;
and in case $v_4y\notin E(C)$, we can assume that $N(y)=\{v_1,v_2,v_3\}$ (because of the symmetry of the 4-wheel),
which yields $C\cong A_3$.
If otherwise $v_1x\notin E(C)$,
then, since there is no 4-clique in $C$, 
we immediately obtain $d(y)=4$ and $v_1,x\in N(y)$,
as $e(C)=11$ and $\delta(C)\geq 3$.
Moreover, $v_4\notin N(y)$, since we otherwise would
obtain a 4-clique, independently of the choice
of the fourth neighbour of $y$.
Thus, we conclude \linebreak $N(y)=\{v_1,v_2,v_3,x\}$
and $C\cong A_3$.

Now, let $v(C)=7$. We distinguish three cases.

{\bf Case 1.} Assume that $C$ contains a subgraph $H \cong K_4$. Let $\{x,y,z\}=V(C)\setminus V(H)=:V'$. 
With $e(C)=13$ and $\delta(C)\geq 3$ it follows that $\{x,y,z\}$ is not an independent set, w.l.o.g.\ $xy\in E(C)$.
By the definition of a collection it further follows that $x$ and $y$ have a common neighbour -- the vertex $z$
or some vertex $v\in V(H)$.

Assume first that $z\in N(x)\cap N(y)$. By $\delta(C)\geq 3$ each vertex in $V'$ needs to have at least one neighbour in
$V(H)$. If there were a matching of size 3 between $V'$ and $V(H)$, then by (b), one of the matching edges
could not be part of a triangle, a contradiction.
If all the three vertices have a common neighbour in $V(H)$, then one easily deduces $C\cong S_2$.
Otherwise, by symmetry we can assume that there is a vertex $v_1\in V(H)$ such that
$v_1x,v_1y\in E(C)$ and $v_1z\notin E(C),$ and moreover, $v_2z\in E(C)$
for some $v_2\in V(H)\setminus \{v_1\}.$ Now, let $\{v_3,v_4\}=V(H)\setminus \{v_1,v_2\}$.
To ensure that $v_2z$ belongs to some triangle in $C$,
we finally need to have exactly one of the edges from $\{v_3z,v_4z,v_2x,v_2y\}$ to be an edge
in $C$. The first two edges however do not result in a triangle collection,
while for the other two edges we get $C\cong S_3.$

Assume then that $z\notin N(x)\cap N(y)$, but $v\in N(x)\cap N(y)$ for some $v\in V(H)$.
Because of (b) and (c), either $xz\in E(C)$ or $yz\in E(C)$, w.l.o.g.\ say $xz\in E(C)$ and  $yz\notin E(C)$.
As $\delta(C)\geq 3$, we then immediately get $yw\in E(C)$ for some
$w\in V(H)\setminus \{v\}$. Moreover, we then need two other edges incident with $z$ besides $xz$,
of which one is $zv$ to ensure that $xz$ belongs to a triangle. If the second edge is $zw$,
then $C\cong S_4$; otherwise $C\cong B_1$.

{\bf Case 2.} Assume that $C$ does not contain a clique of order 4, but there is some $H\subseteq C$
with $H\cong W_4$. Let $\{x,y\}=V(C)\setminus V(H)=:V'$ and
let $z$ be the unique vertex with $d_H(z)=4$. By (b) and (c), it follows that $xy\in E(C)$,
and since $C$ is a collection, there is a common neighbour of $x$ and $y$ in $V(H)$.

Assume first that $z\in N(x)\cap N(y)$. As $\delta(C)\geq 3$, both vertices $x$ and $y$ 
have another neighbour in $V(H)\setminus \{z\}$, however
there cannot be a second common neighbour, since there is no 4-clique in $C$.
One easily checks that $C\cong B_2$ or $C\cong B_3$ follows.

Assume then that $z\notin N(x)\cap N(y)$, but $v\in N(x)\cap N(y)$
for some $v\in V(H)\setminus \{z\}.$ If $xz\in E(C)$ (or $yz\in E(C)$),
we then need $yw\in E(C)$ (or $xw\in E(C)$) for some $w\in N_H(v)\setminus\{z\}$
to ensure that $e(C)=13$ and $\delta(C)\geq 3$ holds
while $C$ is a triangle collection. This gives $C\cong B_4$. Otherwise, we have $z\notin N(x)\cup N(y)$. 
In this case, let $w'$ to be the unique vertex of $H$ not belonging to $N(v)\cup \{v\}$.
Then we also have $w'\notin N(x)\cup N(y)$. Indeed, if we had $yw'\in E(C)$ say,
then as $yw'$ needs to be part of some triangle and as $d(x)\geq 3$ and $e(C)=13$,
we would need $xw'\in E(C)$, in which case it is easily checked that $C$ is not a triangle collection.
So, we can assume that $xv_1\in E(C)$ for some $v_1\in V(H)\setminus \{v,w',z\}$,
and $yv_1\notin E(C)$, because $C$ does not have a 4-clique.
Finally, since $\delta(C)\geq 3$, we need $v_2y\in E(C)$
for the unique vertex $v_2\in V(H)\setminus \{v,w',z,v_1\}$, i.e.\ $C\cong B_5$. 

{\bf Case 3.}  Finally assume that $C$ neither contains a 4-clique nor a 4-wheel.
It is easy to check that $C_0\subseteq C$ (with notation
of vertices as given in Figure \ref{subgraphs}), and by the assumption
of this case we further have $v_1v_3,v_1v_4,v_3v_5\notin E(C)$.
Since $C$ is a triangle collection, we find a vertex $x\in V':=V(C)\setminus V(C_0)$
which belongs to a triangle that also contains an edge $e\in E(C_0)$.
Let $\{y\}=V'\setminus\{x\}$.
By symmetry of $C_0$ we may assume that $e\in\{v_2v_5,v_4v_5,v_1v_5,v_1v_2\}$.

\begin{figure} [htbp]
\centering
\includegraphics[scale=0.6]{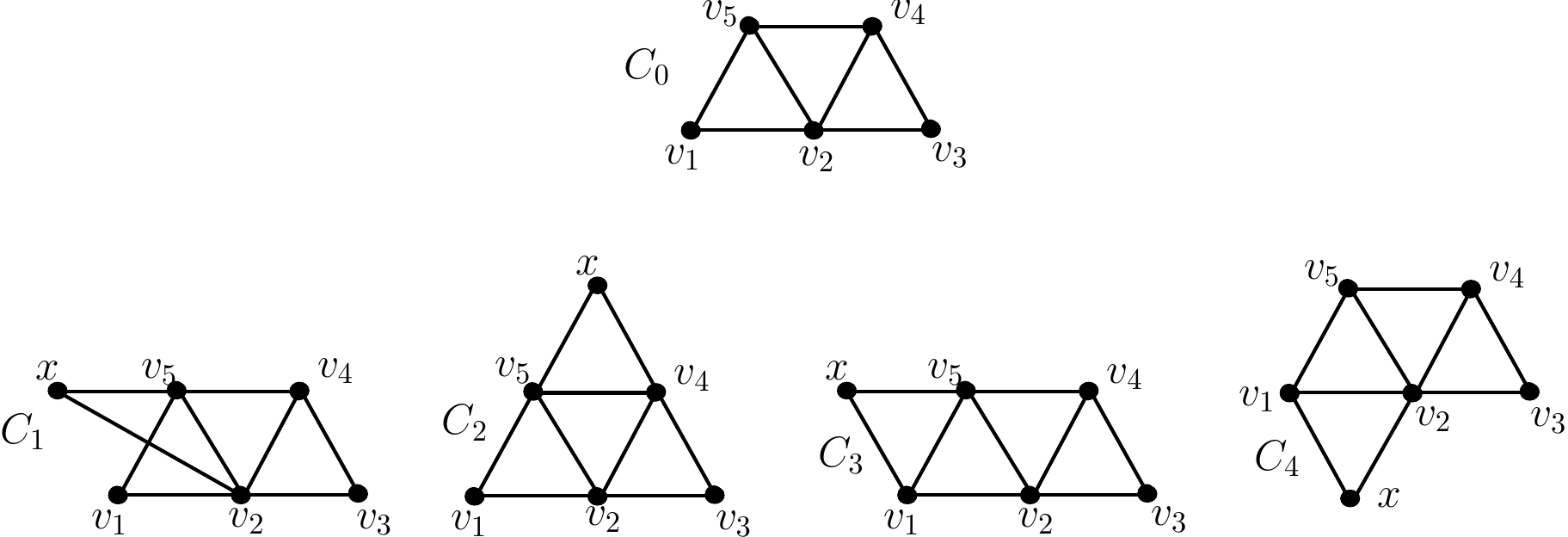}
\caption{Subgraphs.}
\label{subgraphs}
 \end{figure}

Assume first that $e=v_2v_5$ were possible, i.e.\ $C_1\subseteq C$.
Then by assumption of Case 3, every edge in $E(C)\setminus E(C_1)$ would need to be incident
with $y$. Because of (b) and (c) we then had that
$d(y)=4$ and $v_1y,v_3y,xy\in E(C)$. Since these three edges would need
to belong to triangles, we further would need $yv_2\in E(C)$,
which would create a 4-wheel on $V(C)\setminus\{v_3,v_4\}$ with center $v_2$,
in contradiction to the assumption.

So, as next assume that $e=v_4v_5$ were possible, i.e.\ $C_2\subseteq C$.
Then analogously every edge in $E(C)\setminus E(C_2)$ would need to be incident
with $y$, and $d(y)=4$ and $\{v_1,v_3,x\}\subseteq N(y)$, because of (b) and (c). 
But then, independently of what the fourth 
neighbour of $y$ is, one of the edges $v_1y,v_3y,xy$ could not belong to a triangle, again a contradiction.

As third, assume that $e=v_1v_5$, i.e.\ $C_3\subseteq C$.
By the assumption of Case 3,
every edge in $E(C)\setminus (E(C_3)\cup \{xv_3\})$ needs to be incident
with $y$.
If $xv_3\notin E(C)$, then we have $d(y)=4$ and $xy,v_3y\in E(C)$,
because of $e(C)=13$ and $\delta(C)\geq 3$. Depending on how the other two edges incident
with $y$ are chosen, we either obtain a contradiction by creating a 4-clique or a 4-wheel, 
or we see that $C\cong B_6$.
So, let $xv_3\in E(C)$. Then $d(y)=3$, by
(b) and (c), and to have $xv_3$ in a triangle, we need
$yx,yv_3\in E(C)$. It follows that $C\cong B_6$, if $yv_1\in E(C)$ or
$yv_4\in E(C)$, or $C\cong B_7$, if $yv_2\in E(C)$ or $yv_5\in E(C)$.

As last, assume that $e=v_1v_2$, i.e.\ $C_4\subseteq C$. If $xv_3\in E(C)$
were possible, then we had $d(y)=3$ because of $e(C)=13$ and $\delta(C)\geq 3$.
But then, depending on the three edges incident with $y$, we would
get a 4-clique or a 4-wheel in $C$, or we would find an edge which is not contained in a triangle,
a contradiction. So, we can assume that $xv_3\notin E(C)$.
Then, by (b), (c) and the assumption of Case 3,
we deduce that $d(y)=4$ and $yx,yv_3\in E(C)$.
If $yv_2\in E(C)$ were also an edge of $C$, then for any choice of the fourth edge
incident with $y$, we would create a 4-clique or a 4-wheel in $C$.
That is, we can assume that $yv_2\notin E(C)$. But then we need
$v_1y,v_4y\in E(C)$ to ensure that $yx$ and $yv_3$ belong to triangles, which yields $C\cong B_7$.~
\end{proof}

\begin{lemma} \label{breaker_part}
For any collection given by Lemma \ref{collection}, Breaker has a strategy
to prevent cyclic triangles.
\end{lemma}

\begin{proof}
If $C\cong S_i$ for some $i$, note that $C$ is covered by two (not necessarily disjoint) graphs $C(1)$, $C(2)$, plus 
at most one additional edge if $C\cong S_2$, where each of the $C(i)$ is isomorphic to $K_4$ or $W_4$.
Choose edges $a_1$ and $a_2$ as indicated in Figure \ref{special}.
In his first move, Breaker claims the edge $a_1$ if Maker did not orient it before;
otherwise he claims the edge $a_2$.
Afterwards, Breaker plays on $C(1)$ and $C(2)$ separately, 
meaning: each time Maker orients an edge of $C(i)$, Breaker claims an edge of $C(i)$ if there remains one.
Now, using Proposition~\ref{obs:basic} and 
Observation \ref{K_4}, Breaker can do this in a way such that he prevents cyclic triangles on each $C(i)$,
and therefore in $C$. 

Finally, we need to look at the case when $C\cong K_5^-$. By an easy case analysis, it can be proven that Breaker has a strategy
to prevent cyclic triangles on $C$. We give a sketch in the following. Let $V(C)=X\cup Y$ with $X=\{v_1,v_2,v_3\}$ and $Y=\{v_4,v_5\}$, and let 
$E(C)=\binom{X}{2}\cup \{xy:\, x\in X,\, y\in Y\}.$ 

\textbf{Case 1.} Maker orients an edge in $E(X,Y)$ in her first turn.

W.l.o.g.\ let $e=v_1v_4 \in E(X,Y)$ be the edge to which Maker gives an orientation in her first move. 
Then Breaker's strategy is to delete the edge $v_1v_2$. 
Note that $C-\{v_1v_2\}$ is isomorphic to the $4$-wheel $W_4$, here with center $v_3$, and Maker's first arc is not incident with $v_3$. Thus, Breaker can win by Observation \ref{4-wheel}.

\textbf{Case 2.} Maker orients an edge inside $E(X)$ in her first turn.

W.l.o.g.\ let Maker's first oriented edge be $(v_1,v_2)$. Then Breaker's first move will be to delete the edge $v_2v_4$.
Afterwards, Breaker's second move will depend on Makers second move, as follows:

If Maker orients $(v_1,v_3)$ or $(v_3,v_2)$ for her second move, then Breaker claims $v_2v_5$ and afterwards he wins by an easy pairing strategy, with the pairs $\{v_1v_4,v_3v_4\}$ and $\{v_1v_5,v_3v_5\}$.

If Maker for her second move chooses one of the arcs $(v_1,v_4)$, $(v_4,v_1)$, $(v_3,v_4)$, $(v_4,v_3)$, $(v_1,v_5)$, $(v_5,v_2)$, $(v_2,v_3)$ and $(v_3,v_5)$, then Breaker for his second move claims the edge $v_1v_3$. 
As he claims $v_2v_4$ and $v_1v_3$ then, the only triplets on which Maker could create a triangle are $\{v_1,v_2,v_5\}$
and $\{v_2,v_3,v_5\}$. In either of the cases it is easy to check that from now on Breaker can prevent cyclic triangles.

If Maker for her second move chooses $(v_2,v_5)$ or $(v_5,v_3)$, then Breaker claims $v_1v_5$ for his second move. Afterwards there remain three triplets
on which Maker still could create a triangle, namely $\{v_1,v_3,v_4\}$, $\{v_1,v_2,v_3\}$ and $\{v_2,v_3,v_5\}$.
To block a triangle on $\{v_1,v_3,v_4\}$, Breaker can consider a pairing $\{v_1v_4,v_3v_4\}$. For the other two triplets
it is easy to check then that Breaker can prevent cyclic triangles, since the orientation which $v_2v_3$ needs, to create a cyclic triangle, is different for these two remaining triplets.

If Maker for her second move chooses $(v_3,v_1)$, then Breaker needs to claim $v_2v_3$. Afterwards there remain three triplets
on which Maker still could create a triangle, namely $\{v_1,v_3,v_4\}$, $\{v_1,v_2,v_5\}$ and $\{v_1,v_3,v_5\}$.
To block a triangle on $\{v_1,v_3,v_4\}$, Breaker can consider a pairing $\{v_1v_4,v_3v_4\}$. For the other two triplets
it again is easy to check that Breaker can prevent cyclic triangles, since the orientation which $v_1v_5$ needs, to create a cyclic triangle, is different for these two triplets.

Finally, if Maker for her second move chooses $(v_5,v_1)$, then Breaker needs to claim $v_2v_5$. Afterwards there remain three triplets
on which Maker still could create a triangle, namely $\{v_1,v_3,v_4\}$, $\{v_1,v_2,v_3\}$ and $\{v_1,v_3,v_5\}$.
To block a triangle on $\{v_1,v_3,v_4\}$, Breaker can consider a pairing $\{v_1v_4,v_3v_4\}$. For the other two triplets
it again is easy to check that Breaker can prevent cyclic triangles, since the orientation which $v_1v_3$ needs, to create a cyclic triangle, is different for these two triplets.
\end{proof}

To summarize, we have shown now that for $p\ll n^{-\frac{8}{15}}$, a.a.s.\ Breaker can prevent cyclic triangles
in the tournament game on $G\sim \gnp$. Indeed, by Proposition \ref{properties},
Lemma \ref{collection} and Lemma \ref{breaker_part},
we know that there exists no collection $C$ with $m(C)<\frac{15}{8}$
on which Maker has a strategy to create a copy of $T_C$. 
By Proposition \ref{collection_density} we however know that for $p\ll n^{-\frac{8}{15}}$
a random graph $G\sim \gnp$ a.a.s.\ only contains such collections,
and using Observation \ref{reduction} we thus conclude that a.a.s.\
Maker does not have a winning strategy when playing on $G\sim \gnp$,
which at the same time guarantees a winning strategy for Breaker.\hfill $\Box$


\begin{thebibliography}{99}

\bibitem{AS}
N. Alon, J. H. Spencer,
{\bf The Probabilistic Method}, third edition,
Wiley-Interscience Series in Discrete Mathematics and Optimization, 2008.

\bibitem{BeckBook}
J. Beck, {\bf Combinatorial Games: Tic-Tac-Toe Theory}, Cambridge University Press, 2008.

\bibitem{BL} M.\ Bednarska, T.\ \L uczak, Biased positional games for which the random strategies are nearly optimal, \textit{Combinatorica} 20(2000), 477--488.

\bibitem{ChEr}
V.\ Chv\'atal and P.\ Erd\H{o}s, Biased positional games, {\em Annals
of Discrete Mathematics} 2 (1978), 221--228.


\bibitem{CGL} D.\ Clemens, H.\ Gebauer, A.\ Liebenau, The random graph intuition for the tournament game, submitted, arXiv:1307.4229 [math.CO]

\bibitem{ES}
P.\ Erd\H{o}s and J.\ Selfridge, On a combinatorial game, {\em Journal of Combinatorial Theory Series B} 14 (1973), 298--301.



\bibitem{JLR} S. Janson, T.\ \L uczak, A.\ Ruci\'{n}ski, Random graphs, \textit{John Wiley \& Sons, Inc.}, 2000.


\bibitem{Lehman}
A. Lehman, A solution of the Shannon switching game, {\it J. Soc. Indust. Appl. Math.} 12
(1964), 687--725.

\bibitem{MS} T. M\"uller and M. Stojakovi\'c, A threshold for the Maker-Breaker clique game, \textit{Random Structures \& Algorithms} 45 (2014), 318--341.

\bibitem{SS}
M. Stojakovi\'c and T. Szab\'o, Positional games on random graphs, \textit{Random Structures \& Algorithms} 26 (2005), 204--223.

\bibitem{West}
D.\ B.\ West, {\bf Introduction to Graph Theory}, Prentice Hall,
2001.
\end{thebibliography}
\end{document}